\documentclass{amsart}
\usepackage{graphicx}           % Include this line if your 
\usepackage{amsmath} 
\usepackage{amssymb}            % assumes amsmath package installed
\usepackage{amsthm}             % assumes amsmath package installed
\usepackage{enumerate} 
\usepackage{mathdots} 
\usepackage{acronym}

\DeclareMathOperator{\sign}{sign}

\newcommand{\RE}{\mathbb{R}}
\newcommand{\rd}{\mathrm{d}}

\newtheorem{thm}{Theorem}

\acrodef{SMC}{Sliding-Mode Control}
\acrodef{HOSMC}{Higher-Order Sliding-Mode Control}

\begin{document}

\title{Pole-Placement in \\  Higher-Order Sliding-Mode Control}

%\author[1]{Author A\thanks{A.A@university.edu}}
%\author[1]{Author B\thanks{B.B@university.edu}}
%\author[1]{Author C\thanks{C.C@university.edu}}
%\author[2]{Author D\thanks{D.D@university.edu}}
%\author[2]{Author E\thanks{E.E@university.edu}}
%\affil[1]{Department of Computer Science, \LaTeX\ University}
%\affil[2]{Department of Mechanical Engineering, \LaTeX\ University}

%\renewcommand\Authands{ and }

%\thanks[footnoteinfo]{%A brief version of this paper was sent to the European Control Conference, 2013.
%Corresponding author Fernando Casta\~nos.} % Tel. +52 (55) 57 47 37 35.}

\author{Debbie Hern\'andez}
\author{Fernando Casta\~nos}
\address[D. Hern\'andez and F. Casta\~nos]{
 Departamento de Control Autom\'atico, Cinvestav del IPN, M\'exico.
}
\author{Leonid Fridman}
\address[L. Fridman]{Departamento de Ingenier\'ia de Control y Rob\'otica, Divisi\'on de Ingenier\'ia El\'ectrica,
Facultad de Ingenier\'ia, UNAM, M\'exico}

\email[F.~Casta\~nos]{castanos@ieee.org}

\thanks{Research supported by Conacyt, Mexico.}


\begin{abstract}                          % Abstract of not more than 200 words.
 We show that the well-known formula by Ackermann and Utkin can be generalized to the case of higher-order sliding modes.
 By interpreting the eigenvalue assignment of the sliding dynamics as a \emph{zero-placement} problem, the generalization
 becomes straightforward and the proof is greatly simplified. The generalized formula retains the simplicity of the
 original one while allowing to construct the sliding variable of a single-input linear time-invariant system in such a way that it has
 desired relative degree and desired sliding-mode dynamics. The formula can be used as part of a higher-order
 sliding-mode control design methodology, achieving high accuracy and robustness at the same time.
\end{abstract}

\maketitle

%\tableofcontents

\section{Introduction}

\ac{SMC} is by now well known for its robustness properties in the face of
unmatched perturbations and uncertainties~\cite{edwards,utkin}. In the \ac{SMC} approach
the designer first chooses an output with well-defined relative degree and such that the
system is minimum phase. On a second step, the designer devises a control law that drives
the output to zero. Phase minimality then ensures that the system states go to zero along with the output.
A salient feature of \ac{SMC} is that the output (\emph{sliding variable} in the \ac{SMC} literature) is driven exactly
to zero in finite time, even in the presence of matched perturbations.

Conventional \ac{SMC} is restricted to outputs of relative degree equal to one\footnote{This restriction can be
also found, e.g., in passivity based control: It was shown in~\cite{byrnes1991} that a system is feedback
equivalent to a passive system if, and only if, it is minimum phase and its output is of relative degree one.}.
In contrast, modern \ac{SMC} theory (i.e., \ac{HOSMC}) allows for sliding variables with relative degree
higher than one~\cite{levant2003}.

Conventional \ac{SMC} theory is fairly complete in the sense that there exist several methods for choosing a sliding variable with
desired zero dynamics (\emph{sliding-mode dynamics} in the \ac{SMC} literature). One possibility is to put the system in the so-called
regular form and use part of the state as a virtual control that will realize the desired sliding-mode dynamics on a lower dimensional
system~\cite[Sec. 5.1]{utkin}. If the system is single-input, a sliding variable with desired sliding-mode
dynamics can be found without recourse to a coordinate transformation, using the formula by Ackermann and Utkin~\cite{ackermann1998}.
A third possibility is to use the more recent formula presented in~\cite{drazenovic2011}, which works in the multi-input case and also
obviates the need to transform the system into a regular form. Regarding the control law, it is now well-known that a sliding variable of relative degree
one can be robustly driven to zero in finite time by means of a simple unit control with enough gain~\cite[Sec. 3.5]{utkin}.

\subsection{Motivation}

\ac{HOSMC} is under intensive development~\cite{bartolini2003,laghrouche2007,levant2009,orlov,pisano2011,moreno2012}. A major achievement
in this area is the finding of a complete family of sliding-mode controllers that can robustly drive to zero
a sliding variable of arbitrary degree~\cite{levant2005b}. While a high-order sliding variable might appear naturally in specific cases (e.g.,
in the differentiation problem or in the estimation problem), there is at the present
no general design methodology for choosing a sliding variable with prescribed relative degree and prescribed sliding-mode dynamics.
The work reported on this paper is motivated by the need to fill this gap.

Allow us illustrate with the simple chain of integrators
\begin{align*}
 \dot{x}_1 &= x_2 \;, \\
 \dot{x}_2 &= x_3 \;, \\
 \dot{x}_3 &= u + w \;,
\end{align*}
where $x \in \RE^3$ is the state and $u,w \in \RE$ are the control and the unknown perturbation at time $t$ (we omit the time arguments
for ease of notation). Suppose that we want to stabilize the origin.

In the conventional approach one chooses first a sliding
variable of relative degree 1 and such that the associated 2-dimensional
sliding dynamics are stable. Suppose, e.g., we desire sliding dynamics having an eigenvalue $-1$ with multiplicity 2. We can use
the well-known formula by Ackermann and Utkin to obtain the sliding variable $\sigma = x_1 + 2x_2 + x_3$. Finally, we can apply the control
law $u = - x_2 - 2x_3 - \bar{w}\sign(\sigma)$, where $\bar{w}$ is a known upper bound for $|w|$. It is not hard to see that the trajectories
converge globally asymptotically to zero, regardless of $w$.

Consider now the case $\sigma = x_1$. The relative degree of $\sigma$ is equal to the system's dimension, so there are no sliding 
dynamics to worry about. It is by now a standard result of \ac{HOSMC} theory that the (substantially more complex) controller 
\begin{equation} \label{eq:mot}
 u = -k_0 \frac{\ddot{\sigma} + 2(|\dot{\sigma}|+|\sigma|^{2/3})^{-1/2}(\dot{\sigma}+|\sigma|^{2/3}\sign(\sigma))}
  {|\ddot{\sigma}| + 2(|\dot{\sigma}|+|\sigma|^{2/3})^{1/2}} \;,
\end{equation}
with $\alpha > 0$ high enough, drives the state to zero in finite time, regardless of $w$.

The computation of a sliding variable of relative degree equal to the dimension of the plant was simple because
the system is in a canonical form. This suggests that, for a general linear controllable system, we first put it in controller canonical form
and then take the state with highest relative degree as the sliding variable. In this way, the extreme case of relative degree equal to the
system's dimension (no sliding dynamics) can be covered systematically. The other extreme case, that of relative degree 1 (sliding dynamics
of codimension 1), can be covered using Ackermann and Utkin's formula. Note, however, that there is no systematic method for constructing
a sliding variable of intermediate relative degree (in our example, of relative degree 2). To such a sliding variable there would correspond
a sliding dynamics of dimension 1. This dynamics can be enforced with a controller much simpler than~\eqref{eq:mot}, thus arriving at a fair
compromise between order reduction and controller complexity.

\subsection{Contribution}

Our main contribution, Theorem~\ref{thm:main}, concerns single-input linear time-invariant (LTI) systems. The selection of the sliding variable
is interpreted as a zero-placement problem, which allows us to generalize the formula of Ackermann and Utkin to the case of arbitrary relative
degree. Our proof is simpler (more insightful) than the proof of the original problem. The formula makes it possible for the designer to construct a sliding
variable with desired sliding-mode dynamics of arbitrary dimension.

For the case of relative degree 2 in our motivational example above, application of Theorem~\ref{thm:main} to a sliding dynamics
with desired eigenvalue -1 gives the sliding surface $\sigma = x_1 + x_2$. The sliding dynamics can be enforced, e.g., with the twisting
controller $u = -x_3 - k_0\sign(\sigma) - k_1\sign(\dot{\sigma})$, where $k_0$ and $k_1$ are high enough to reject $w$.  

\subsection{Paper Structure}

In the following section we give some preliminaries on relative degree, zero dynamics and \ac{SMC}. The section is included mainly
to set up the notation and to provide some context for our main result, which is contained in Section~\ref{sec:main}.
Section~\ref{sec:exa} provides a thorough example and the conclusions are given in Section~\ref{sec:conc}.

\section{Preliminaries}

Consider the LTI system
\begin{subequations} \label{eq:LTI}
\begin{equation} 
 \dot{x} = Ax + B(u+w) \;, \quad x \in \RE^n \;, \quad u,w \in \RE \;,
\end{equation}
where $x$ is the state, $u$ the control and $w$ the unknown perturbation at time $t$ (we omit the time arguments).
The pair $(A,B)$ is assumed to be controllable. Suppose that we want to steer $x$ to zero despite the presence of $w$.
The problem can be approached in two steps: First, find a `virtual' output 
\begin{equation}
 \sigma = Cx \;, \quad \sigma \in \RE
\end{equation}
such that $\sigma \equiv 0$ implies $x \to 0$ as $t \to \infty$. Next, design a feedback control law that ensures that
$\sigma \to 0$ either as $t \to \infty$ or as $t \to  T$, $T > 0$, depending on the desired degree of smoothness and robustness of the controller.
\end{subequations}

\subsection{Relative degree and zero dynamics}

Recall that~\eqref{eq:LTI} is said to have \emph{relative degree} $r$ if
$C A^{i-1} B = 0$, $1 \leq i < r$ and $C A^{r-1} B \neq 0$.
If~\eqref{eq:LTI} has relative degree $r$, then it is possible to take $\sigma$ and its successive 
$r-1$ time-derivatives as a partial set of coordinates $\xi_1,\dots,\xi_r$. More precisely,
there exists a full-rank matrix $B^\perp \in \RE^{(n-r)\times n}$ such that $B^\perp B = 0$ and
\begin{displaymath}
 \left[
 \begin{array}{c}
  \eta \\ \hline \xi
 \end{array}
 \right] = %Tx
 \left[
 \begin{array}{c}
  B^\perp \\ \hline C \\ \vdots \\ CA^{r-1}
 \end{array}
 \right] x
\end{displaymath}
%such that,
is a coordinate transformation, that is, $T$ is invertible~\cite[Prop. 4.1.3]{isidori}. It is straightforward to
verify that, in the new coordinates, system~\eqref{eq:LTI} takes the \emph{normal form}
\begin{subequations}
\begin{align}
 \left[
 \begin{array}{c}
  \dot{\eta} \\ \hline \dot{\xi}_1 \\ \vdots \\ \dot{\xi}_{r-1} \\ \dot{\xi}_r
 \end{array}
 \right]
 &= 
 \left[
 \begin{array}{c}
  A_0 \eta + B_0 \xi \\ \hline \xi_2 \\ \vdots \\ \xi_r \\ CA^{r}x
 \end{array}
 \right] + 
 \left[
 \begin{array}{c}
  0 \\ \hline 0 \\ \vdots \\ 0 \\ CA^{r-1}B 
 \end{array}
 \right] (u + w) \label{eq:nfs} \\
 \sigma &= \xi_1 \;.
\end{align}
\end{subequations}
The dynamics $\dot{\eta} = A_0 \eta$, $\eta \in \RE^{n-r}$, are the \emph{zero dynamics}. It is well known~\cite[Ex. 4.1.3]{marino} that the eigenvalues
of $A_0$ coincide with the zeros of the transfer function
\begin{displaymath}
 g(s) = C(sI-A)^{-1}B \;.
\end{displaymath}
If the zeros of $g(s)$ have real part strictly less than zero, we say that the system is \emph{minimum phase}. Thus, we can
reformulate our first step as: \emph{find a virtual output such that~\eqref{eq:LTI} has stable zeros at desired locations}.

\subsection{Sliding-mode control}

If $w$ is majored by a known bound, then the robust stabilization objective can be accomplished using nonsmooth
control laws (solutions of differential equations with discontinuous right-hands are taken in Fillipov's sense). In
conventional first-order \ac{SMC}~\cite{utkin}, the search for $\sigma$ is confined to outputs of relative degree one
and the control takes the form\footnote{At the expense of the loss of global stability and a higher gain $k_0$, the
term $CAx$ is sometimes omitted by incorporating it into $w$.}
\begin{equation} \label{eq:ford}
 u = -\frac{CAx + k_0 \sign(\xi_1)}{CB} 
\end{equation}
with $k_0 > |w|$. This control law guaranties that $\sigma$ will reach zero in a finite time $T$ and will stay at zero
for all future time, regardless of the presence of $w$. The matrix $C$ can be set using the formula by Ackermann and Utkin,
recalled in the following theorem.

\begin{thm}[\cite{ackermann1998}] \label{thm:AkUt}
Let $e_1 := [\begin{matrix} 0 & 0 & \cdots & 0 & 1 \end{matrix}]$ and let $P$ be the system's controllability matrix.
If $C = e_1 P^{-1}\beta(A)$ with $\beta(\lambda) = \lambda^{n-1} + \beta_{n-2}\lambda^{n-2} + \cdots + \beta_{1}\lambda + \beta_{0}$, then
the roots of $\beta(\lambda)$ are the eigenvalues of the sliding-mode dynamics in the plane $\sigma = 0$.
\end{thm}

Restated somewhat differently, this theorem says that the virtual output $\sigma = Cx$ results in a zero dynamics
with characteristic polynomial $\beta(\lambda)$, so that the roots of $\beta(\lambda)$ are precisely the
eigenvalues of $A_0$. Thus, by choosing a Hurwitz polynomial $\beta(\lambda)$ we ensure that~\eqref{eq:LTI} is minimum phase,
which implies that all the states converge to the origin when $\sigma = \xi_1$ is constrained to zero. In view of the
previous discussion, this amount to saying that the roots of $\beta(\lambda)$ coincide with the zeros of $g(s)$. This is
the key observation that will allow us to generalize the Theorem while simplifying its proof.

Modern sliding-mode control theory considers the more general case of relative degree $r \ge 1$.
Suppose, for example, that~\eqref{eq:LTI} has relative degree $r = 2$. The second-order \emph{twisting}
controller
\begin{displaymath}
 u = -\frac{CA^2x + k_0 \sign(\xi_1) + k_1 \sign(\xi_2)}{CAB} 
\end{displaymath}
with $k_1 > |w|$ and $k_0 > k_1 + |w|$ will drive $\sigma = \xi_1$ and $\dot{\sigma} = \xi_2$ to zero in finite time
(again, regardless of $w$). More generally,
We say that an \emph{$r$-sliding mode} occurs whenever the successive time derivatives $\sigma,\dot{\sigma},\dots,\sigma^{(r-1)}$
are continuous functions of the closed-loop state-space variables and $\sigma = \dot{\sigma} = \cdots = \sigma^{(r-1)} = 0$
(i.e., $\xi = 0$). Nowadays, it is possible to construct a controller of the form
\footnote{Actually, in the original
version~\cite{levant2003}, the nonlinear counterpart of $CA^rx$ is regarded as a perturbation and omitted from the equation. Since we
are dealing with simple linear systems, we have included it to reduce the necessary gains.}
\begin{equation} \label{eq:arbr}
 u = -\frac{CA^rx + f(\xi)}{CA^{r-1}B}
\end{equation}
enforcing an $r$-sliding mode for arbitrary $r$, though it is worth mentioning that the complexity of $f(\xi)$ increases rapidly as $r$
increases (see~\cite{levant2003} for details).

One can think of at least two circumstances that justify the increased complexity of higher-order sliding-mode
controllers: prescribed degree of smoothness and prescribed order of accuracy in the face of unmodeled dynamics and
controller discretization.

Regarding smoothness, suppose that~\eqref{eq:LTI} has relative degree $r$ and suppose that a chain of $k$ integrators is cascaded to the system input,
$u_k := u^{(k)}$.
\begin{displaymath}
 u = \int_{\tau_k = 0}^t\cdots\int_{\tau_1 = 0}^{\tau_2} u_k\rd \tau_1\cdots \rd \tau_k \;.
\end{displaymath}
The relative degree of the system with new input $u_k$ and output $\sigma$ is $r+k$. Now, an $(r+k)$-sliding mode has to be enforced
by $u_k$, but the true input $u$ is at least $k-1$ times continuously differentiable.

Regarding order accuracy, it is probably best to recall the following theorem.
\begin{thm}[\cite{levant2005b}]
Let the control value be updated at the moments $t_{i}$, with $t_{i+1}-t_{i} = \tau = \mathrm{const} >0$;
$t\in \lbrack t_{i},t_{i+1})$ (the discrete sampling case). Then, controller~\eqref{eq:arbr} provides in finite time for keeping
the inequalities 
\begin{equation} \label{eq:accu}
 \left\vert      \sigma     \right\vert < \mu_{0}\tau^{r} \;, \; 
 \left\vert \dot{\sigma}    \right\vert < \mu_{1}\tau^{r-1} \;, \; \dots \; , \;
  \left\vert \sigma^{(r-1)} \right\vert < \mu_{r-1}\tau
\end{equation}
with some positive constants $\mu_{0},\mu_{1},\dots,\mu_{r-1}$.
\end{thm}
(See~\cite{levant2005b} for the specific form of $f(\xi)$ in~\eqref{eq:arbr}.)
It is also shown in~\cite{levant2010b} that in the presence of an actuator of the form
$\tau \dot{z} = a(z,u)$, $v = v(z)$, $z \in \RE^m$, $v \in \RE$ with $u$ the input of the actuator,
$v$ its output and $\mu$ the time constant, inequalities~\eqref{eq:accu} also hold under reasonable
assumptions.

\section{Main result} \label{sec:main}

We have recalled in the previous section that, for arbitrary $r$, it is possible to enforce an $r$-sliding motion despite the presence
of perturbations. Now we show that Theorem~\ref{thm:AkUt} holds for arbitrary relative degree, so it can be used to select a virtual output
with desired relative degree and desired sliding-mode dynamics. The zero dynamics interpretation allows for a simpler proof.

\begin{thm} \label{thm:main}
If
\begin{equation} \label{eq:main}
 C = e_1 P^{-1}\gamma(A) \;,
\end{equation}
with $\gamma(\lambda) = \lambda^{n-r} + \gamma_{n-r-1}\lambda^{n-r-1} + \cdots + \gamma_{1}\lambda + \gamma_{0}$, 
then $\sigma$ is of relative degree $r$ and the roots of $\gamma(\lambda)$ are the eigenvalues of the sliding-mode dynamics in the 
intersection of the planes $\sigma = \dot{\sigma} = \cdots = \sigma^{(r-1)} = 0$.
\end{thm}

\begin{proof}
Let us assume that the system is given in controller canonical form with system matrices $\hat{A}$ and $\hat{B}$.
To verify~\eqref{eq:main}, we will show that for $\hat{C} = e_1 \hat{P}^{-1}\gamma(\hat{A})$,
the numerator of $g(s) = \hat{C}(sI - \hat{A})^{-1}\hat{B}$ is equal to $\gamma(s)$.

It is a standard result that, for a system in controller canonical form, we have~\cite{williams}
\begin{equation} \label{eq:As}
\begin{split}
e_1 \hat{P}^{-1} &= \begin{bmatrix} 1 & 0 & \cdots & 0 & 0 \end{bmatrix} \\
 \begin{bmatrix}
  1 & 0 & \cdots & 0 & 0
 \end{bmatrix} \hat{A} \quad \, &= 
 \begin{bmatrix}
  0 & 1 & \cdots & 0 & 0
 \end{bmatrix} \\
                       & \; \vdots  \\
 \begin{bmatrix}
  1 & 0 & \cdots & 0 & 0
 \end{bmatrix} \hat{A}^{n-2} &= 
 \begin{bmatrix}
  0 & 0 & \cdots & 1 & 0
 \end{bmatrix} \\
 \begin{bmatrix}
  1 & 0 & \cdots & 0 & 0
 \end{bmatrix} \hat{A}^{n-1} &= 
 \begin{bmatrix}
  0 & 0 & \cdots & 0 & 1
 \end{bmatrix} \;.
\end{split}
\end{equation}
It then follows that
\begin{displaymath}
\hat{C} = \begin{bmatrix} \gamma_{0} & \gamma_{1} & \cdots & \gamma_{n-r-1} & 1 & 0 & \cdots & 0 \end{bmatrix} \;.
\end{displaymath}

Since $\hat{A},\hat{B}$ and $\hat{C}$ are in controller canonical form, the transfer function is simply
\begin{displaymath}
 g(s) = \frac{s^{n-r} + \gamma_{n-r-1}s^{n-r-1} + \cdots + \gamma_{1}s + \gamma_{0}}{s^n + a_{n-1}s^{n-1} + \cdots + a_1s + a_0} \;,
\end{displaymath}
which shows that the relative degree is $r$. Since the numerator is equal to $\gamma(s)$, the eigenvalues of the sliding-mode dynamics
are equal to the roots of $\gamma(s)$.

Now, to address the general case, consider the transformation $T = P\hat{P}^{-1}$, which is such that
$\hat{A} = T^{-1} A T$. We have $C = \hat{C}T^{-1}$, that is, $C = e_1 \hat{P}^{-1}\gamma(\hat{A})T^{-1}$.
Finally, from $\hat{P}^{-1} = P^{-1}T$ and $\gamma(\hat{A}) = T^{-1}\gamma(A)T$ we recover~\eqref{eq:main}.
\end{proof}

\section{Example} \label{sec:exa}

Consider the linearized model of a real inverted pendulum on a cart~\cite{fantoni}
\begin{equation}  \label{eq:sys}
 \dot{x} = 
  \begin{bmatrix}
    0 & 1 &     0 & 0 \\
    0 & 0 & -1.56 & 0 \\
    0 & 0 &     0 & 1 \\
    0 & 0 & 46.87 & 0
  \end{bmatrix}
  x + 
  \begin{bmatrix}
   0 \\ 0.97 \\ 0 \\ -3.98
  \end{bmatrix}
  (u + w) \;,
\end{equation}
where $x_1$, $x_2$, $x_3$ and $x_4$ are the position and velocity of the cart, and the angle and angular velocity of the pole, respectively.
The system is controllable and the open-loop characteristic polynomial is $\lambda^2(\lambda+6.85)(\lambda-6.85)$. Suppose that we want to regulate
the state to zero, in spite of any perturbations satisfying the bound $|w| \le 1$.

\subsection{First-order sliding mode control}

\begin{figure}
\begin{center}
 \includegraphics[width = 0.80\textwidth]{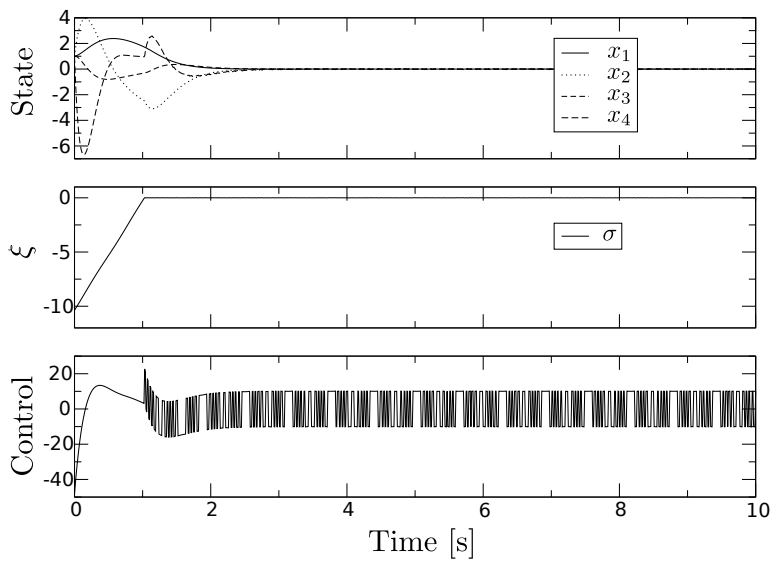}
\end{center}
\caption{Simulation results for a first-order sliding mode controller. The system is perturbed by
$w = 0.5\sin(10 t)$. The control~\eqref{eq:ford} is sampled and held every $\tau = 0.001$ seconds.}
\label{fig:1ordx}
\end{figure}

Consider the problem of designing a first-order sliding mode controller with sliding-mode dynamics having eigenvalues
$z_i = -5$, $i = 1,2,3$. Applying~\eqref{eq:main} with $\gamma(\lambda) = (\lambda+5)^3$ gives
\begin{displaymath}
 C =
  \begin{bmatrix}
   -3.2002 & -1.9201 & -4.5411 & -0.7166
  \end{bmatrix} \;,
\end{displaymath}
which in turn yields the expected transfer function
\begin{displaymath}
 g(s) = C\left(sI - A\right)^{-1}B = \frac{\left( s+5 \right)^3}{s^2(s+6.85)(s-6.85)} \;.
\end{displaymath}
To enforce a sliding motion on the surface $\sigma = 0$ we apply the control~\eqref{eq:ford} with
$k_0 = 10$. Fig.~\ref{fig:1ordx} shows the simulated response when
\begin{displaymath}
 w = \sin(10t) \quad  \text{and} \quad
  x_0^\top = 
  \begin{bmatrix}
   1 & 1 & 1 & 1
  \end{bmatrix} 
\end{displaymath}
and the control law is sampled and held every $\tau = 0.001$ seconds. It can be seen that, once the
state reaches the sliding surface, the state converges exponentially to the origin, despite $w$.

\begin{figure}
\begin{center}
 \includegraphics[width = 0.80\textwidth]{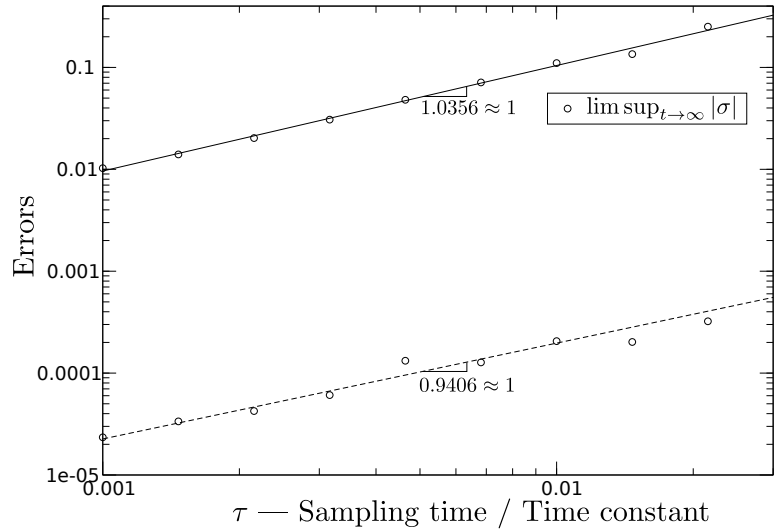}
\end{center}
\caption{First-order sliding mode. Order of the error as a function of the sampling period (solid) and
 the actuator time-constant (dashed). The error order matches the sliding-mode order almost exactly.}
\label{fig:1orde}
\end{figure}

To verify the order of accuracy established in~\eqref{eq:accu}, we take logarithms on
both sides of the inequalities (the base is not important),
\begin{displaymath}
 \log(|\sigma^{(i)}|) < \log(\mu_i) + (r-i)\log(\tau) \;, \quad i = 0,\dots,r-1 \;.
\end{displaymath}
Notice that, on a logarithmic scale, the right-hand is a straight line with slope $r-i$ and ordinate at the
origin $\log(\mu_i)$. To verify that the order of the error $|\sigma^{(i)}|$ as a function of $\tau$ is indeed
$r-i$, the closed-loop system was simulated for several values of $\tau$, both for a zero order hold with sampling
period $\tau$ and for a (previously neglected) actuator of the form $\tau \dot{v} = -v + u$.
We recorded the maximum error after the transient, $\limsup_{t \to \infty} |\sigma^{(i)}|$. The best linear
interpolation on a least square sense was then computed to recover an estimate of $\log(\mu_i)$ and $r-i$.
Fig.~\ref{fig:1orde} shows that the estimations agree well with~\eqref{eq:accu}.

\subsection{Second-order sliding mode control}

Suppose now that we desire a sliding-mode dynamics with eigenvalues $z_i = -5$, $i = 1,2$.
Applying~\eqref{eq:main} with $\gamma(\lambda) = (\lambda+5)^2$ gives
\begin{displaymath}
 C =
  \begin{bmatrix}
   -0.6400 & -0.2560 & -0.4062 & -0.0621
  \end{bmatrix} 
\end{displaymath}
and
\begin{displaymath}
 g(s) = \frac{\left( s+5 \right)^2}{s^2(s+6.85)(s-6.85)} \;.
\end{displaymath}

\begin{figure}
\begin{center}
 \includegraphics[width = 0.80\textwidth]{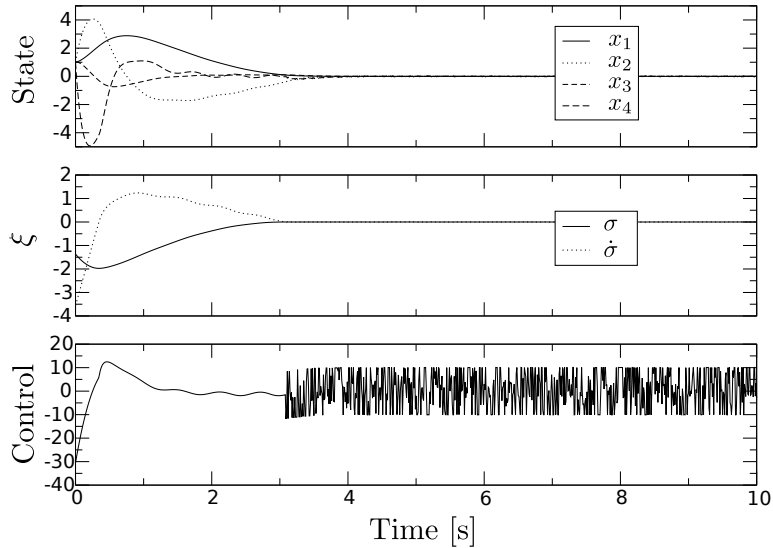}
\end{center}
\caption{Simulation results for a second-order sliding mode controller. The system is perturbed by
$w = 0.5\sin(10 t)$. The control~\eqref{eq:sord} is sampled and held every $\tau = 0.001$ seconds.}
\label{fig:2ordx}
\end{figure}

\begin{figure}
\begin{center}
 \includegraphics[width = 0.80\textwidth]{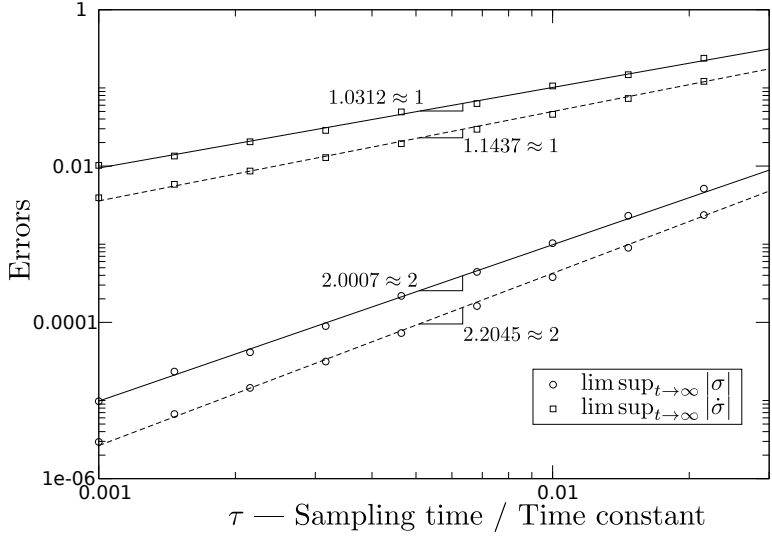}
\end{center}
\caption{Second-order sliding mode. Order of the errors as functions of the sampling period (solid)
 and the actuator time-constant (dashed). The error order for $\sigma$ matches well with the sliding-mode order.}
\label{fig:2orde}
\end{figure}

To enforce a second-order sliding motion on the surface $\sigma = \dot{\sigma} = 0$ we apply the
control~\eqref{eq:arbr} with $f(\xi)$ as in~\cite{levant2005b}, that is,
\begin{equation} \label{eq:sord}
 u = -\frac{1}{CAB}\left( CA^2 x +  10 \frac{\dot{\sigma} + |\sigma|^{1/2}\sign(\sigma)}{|\dot{\sigma}| + |\sigma|^{1/2}} \right) \;.
\end{equation}
Fig.~\ref{fig:2ordx} shows the simulated response for the same perturbation, initial conditions
and sampling time as before. It can be seen that, once the state reaches the sliding surface, the
state converges exponentially to the origin, again despite $w$. Fig.~\ref{fig:2orde} shows
the system accuracy for several sampling times and several actuator time-constants. Inequality~\eqref{eq:accu} is again verified.

\subsection{Third-order sliding mode control}

Consider the problem of designing a third-order sliding mode controller with sliding-mode dynamics having the eigenvalue 
$z_{1} = -5$. Applying~\eqref{eq:main} with $\gamma(\lambda) = \lambda + 5$ gives
\begin{displaymath}
 C =
  \begin{bmatrix}
   -0.1280 & -0.0256 & -0.0310 & -0.0062
  \end{bmatrix} \;,
\end{displaymath}
which in turn yields the expected transfer function
\begin{displaymath}
 g(s)  = C\left(sI - A\right)^{-1}B = \frac{s+5}{s^2(s+6.85)(s-6.85)} \;.
\end{displaymath}

\begin{figure}
\begin{center}
 \includegraphics[width = 0.80\textwidth]{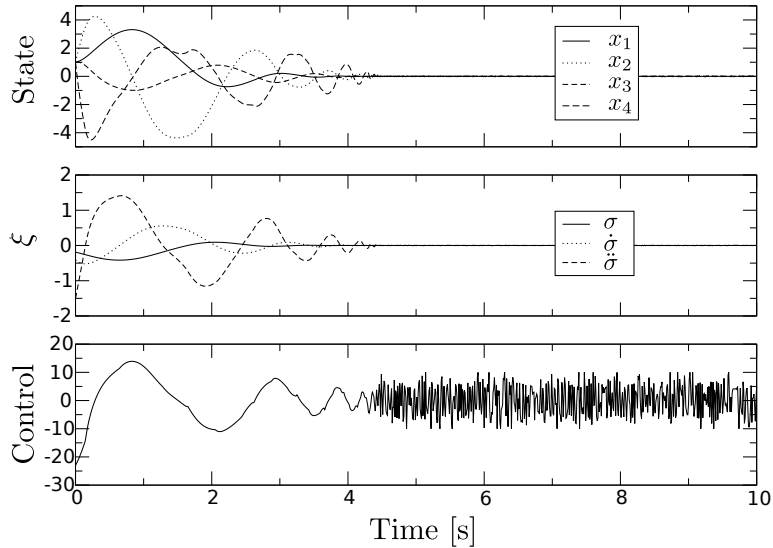}
\end{center}
\caption{Simulation results for a third-order sliding mode controller. The system is perturbed by
$w = 0.5\sin(10 t)$. The control~\eqref{eq:tord} is sampled and held every $\tau = 0.001$ seconds.}
\label{fig:3ordx}
\end{figure}

To enforce a third-order sliding motion on the surface $\sigma = \dot{\sigma} = \ddot{\sigma} = 0$ we
%apply~\eqref{eq:arbr} with $f(\xi)$ as in~\cite{levant2005b}.
apply
%For $r = 3$, we use the controller
\begin{equation} \label{eq:tord}
 u = -\frac{1}{CA^2B}\Bigg( CA^3 x +
  10 
  \frac{\ddot{\sigma} + 2(|\dot{\sigma}|+|\sigma|^{2/3})^{-1/2}(\dot{\sigma}+|\sigma|^{2/3}\sign(\sigma))}
  {|\ddot{\sigma}| + 2(|\dot{\sigma}|+|\sigma|^{2/3})^{1/2}}
 \Bigg) \;.
\end{equation}
Fig.~\ref{fig:3ordx} shows the simulated response for the same perturbation, initial conditions
and sampling time as before. Again, the state converges exponentially to the origin once the state reaches
the sliding surface, despite $w$.

\begin{figure}
\begin{center}
 \includegraphics[width = 0.80\textwidth]{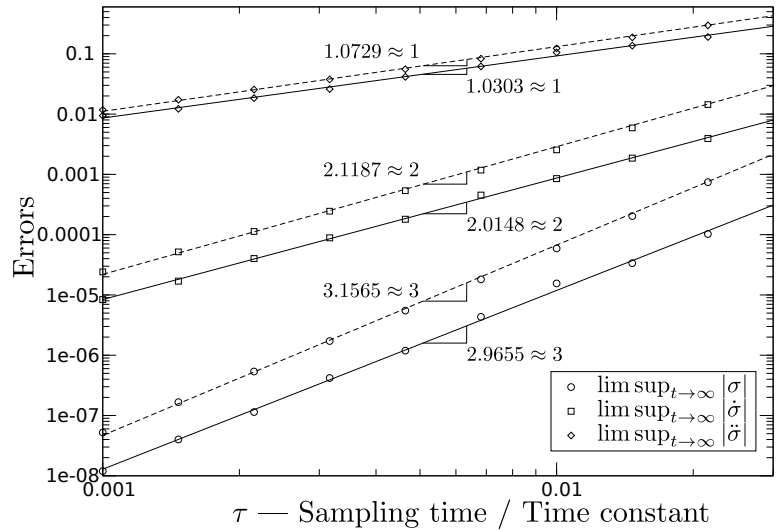}
\end{center}
\caption{Third-order sliding mode. Order of the errors as functions of the sampling period (solid)
 and the actuator time-constant (dashed). The error order for $\sigma$ matches well with the sliding-mode order.}
\label{fig:3orde}
\end{figure}

\section{Conclusions} \label{sec:conc}

We have presented a generalization of the well-known formula by Ackermann and Utkin. A complete design
cycle can now be easily carried out. Formula~\eqref{eq:main} allows the control designer to first specify
a desired sliding-dynamics of any order. Then, the sliding-mode dynamics can be enforced using the corresponding
higher-order sliding mode controller given in~\cite{levant2005b}.

It is clear that there is a trade-off between complexity of a sliding mode controller, accuracy and order
reduction of the equations of motion. By being able to choose the relative degree of the system, the designer can now
decide on the right compromise, depending on the particular application at hand.
 
We have used the notion of accuracy in the face of sample and hold as our main motivation for using
higher-order SMC but other criteria, such as smoothness, can also prompt the use of higher-order SMC.

\bibliographystyle{IEEEtran}
%\bibliographystyle{plain}        % Include this if you use bibtex 
%\bibliography{implicitSM}        % and a bib file to produce the 
                                 % bibliography (preferred). The
                                 % correct style is generated by
                                 % Elsevier at the time of printing.
%\bibliography{IEEEabrv,/home/papa/bibfile}
\bibliography{IEEEabrv,PP_HOSM}

\begin{thebibliography}{10}
\providecommand{\url}[1]{#1}
\csname url@samestyle\endcsname
\providecommand{\newblock}{\relax}
\providecommand{\bibinfo}[2]{#2}
\providecommand{\BIBentrySTDinterwordspacing}{\spaceskip=0pt\relax}
\providecommand{\BIBentryALTinterwordstretchfactor}{4}
\providecommand{\BIBentryALTinterwordspacing}{\spaceskip=\fontdimen2\font plus
\BIBentryALTinterwordstretchfactor\fontdimen3\font minus
  \fontdimen4\font\relax}
\providecommand{\BIBforeignlanguage}[2]{{%
\expandafter\ifx\csname l@#1\endcsname\relax
\typeout{** WARNING: IEEEtran.bst: No hyphenation pattern has been}%
\typeout{** loaded for the language `#1'. Using the pattern for}%
\typeout{** the default language instead.}%
\else
\language=\csname l@#1\endcsname
\fi
#2}}
\providecommand{\BIBdecl}{\relax}
\BIBdecl

\bibitem{edwards}
C.~Edwards and S.~K. Spurgeon, \emph{Sliding mode control: theory and
  applications}.\hskip 1em plus 0.5em minus 0.4em\relax Padstow, UK: CRC, 1998.

\bibitem{utkin}
V.~Utkin, J.~Guldner, and J.~Shi, \emph{Sliding Modes in Electromechanical
  Systems}.\hskip 1em plus 0.5em minus 0.4em\relax London, U.K.: Taylor \&
  Francis, 1999.

\bibitem{byrnes1991}
C.~I. Byrnes and A.~Isidori, ``Asymptotic stabilization of minimum phase
  nonlinear systems,'' \emph{{IEEE} Trans. Autom. Control}, vol.~36, pp.
  1122--1137, Oct. 1991.

\bibitem{levant2003}
A.~Levant, ``Higher-order sliding modes, differentiation and output-feedback
  control,'' \emph{Int. J. Control}, vol.~76, pp. 924--941, 2003.

\bibitem{ackermann1998}
J.~Ackermann and V.~Utkin, ``Sliding mode control design based on {Ackermann's}
  formula,'' \emph{{IEEE} Trans. Autom. Control}, vol.~43, pp. 234 -- 237, Feb.
  1998.

\bibitem{drazenovic2011}
B.~Dra{\v z}enovi\'c, {\v C}.~Milosavljevi{\' c}, B.~Veseli{\' c}, and
  V.~Gligi{\' c}, ``Comprehensive approach to sliding subspace design in linear
  time invariant systems,'' in \emph{Proc. of the Variable Structure Systems
  Workshop}, Mumbai, India, Jan. 2012, pp. 473 -- 478.

\bibitem{bartolini2003}
G.~Bartolini, A.~Pisano, E.~Punta, and E.~Usai, ``A survey of applications of
  second-order sliding mode control to mechanical systems,'' \emph{Int. J.
  Control}, vol. 76:9-10, pp. 875 -- 892, 2003.

\bibitem{laghrouche2007}
S.~Laghrouche, F.~Plestan, and A.~Glumineaub, ``Higher order sliding mode
  control based on integral sliding mode,'' \emph{Automatica}, vol.~43, pp. 531
  -- 537, 2007.

\bibitem{levant2009}
A.~Levant and A.~Michael, ``Adjustment of high-order sliding-mode
  controllers,'' \emph{Int. J. Robust Nonlinear Control}, vol.~19, pp.
  1657--1672, 2009.

\bibitem{orlov}
Y.~Orlov, \emph{Discontinuous Systems, Lyapunov Analysis and Robust Synthesis
  under Uncertainty Conditions}.\hskip 1em plus 0.5em minus 0.4em\relax London:
  Springer-Verlag, 2009.

\bibitem{pisano2011}
A.~Pisano and E.~Usai, ``Sliding mode control: A survey with applications in
  math,'' \emph{Mathematics and Computers in Simulation}, vol.~81, pp. 954 --
  979, 2011.

\bibitem{moreno2012}
J.~A. Moreno and M.~Osorio, ``Strict {Lyapunov} functions for the
  super-twisting algorithm,'' \emph{{IEEE} Trans. Autom. Control}, vol.~57, pp.
  1035 -- 1040, Apr. 2012.

\bibitem{levant2005b}
A.~Levant, ``Quasi-continuous high-order sliding-mode controllers,''
  \emph{{IEEE} Trans. Autom. Control}, vol.~50, pp. 1812 -- 1816, Nov. 2005.

\bibitem{isidori}
A.~Isidori, \emph{Nonlinear Control Systems}.\hskip 1em plus 0.5em minus
  0.4em\relax London, U.K.: Springer-Verlag, 1996.

\bibitem{marino}
R.~Marino and P.~Tomei, \emph{Nonlinear Control Design: Geometric, Adaptive,
  and Robust}.\hskip 1em plus 0.5em minus 0.4em\relax Prentice-Hall, 1995.

\bibitem{levant2010b}
A.~Levant, ``Chattering analysis,'' \emph{{IEEE} Trans. Autom. Control},
  vol.~55, pp. 1380 -- 1389, Jun. 2010.

\bibitem{williams}
R.~L. Williams and D.~A. Lawrence, \emph{Linear state-space control
  systems}.\hskip 1em plus 0.5em minus 0.4em\relax New Jersey: John Wiley \&
  Sons, Inc., 2007.

\bibitem{fantoni}
I.~Fantoni and R.~Lozano, \emph{Non-Linear Control for Underactuated Mechanical
  Systems}.\hskip 1em plus 0.5em minus 0.4em\relax London: Springer-Verlag,
  2002.

\end{thebibliography}

\end{document}